\documentclass{amsart}
\usepackage{amsfonts,amsmath,amssymb,graphicx,graphics,color,
lscape,comment,amsbsy,mathrsfs}

\numberwithin{equation}{section}

\newtheorem{theorem}{Theorem}
\numberwithin{theorem}{section}
\newtheorem{lemma}{Lemma}
\numberwithin{lemma}{section}
\newtheorem{prop}{Proposition}
\numberwithin{prop}{section}
\newtheorem{corol}{Corollary}
\numberwithin{corol}{section}
\newtheorem{remark}{Remark}
\numberwithin{remark}{section}

\numberwithin{defi}{section}

\numberwithin{exe}{section}

\newcommand{\N}{\mathbb{N}}

\newcommand{\notthis}[1]{}

\title[Inversion Formulas]{\textbf{Inversion Formulas with Hypergeometric polynomials and its application to an integral equation}}
\author{R. Nasri(*), A. Simonian(*) and F. Guillemin (**)}
\address{Address:  
(*) Orange Labs, OLN/NMP, Orange Gardens, 
44 avenue de la République, CS 50010, 92326 Chatillon Cedex, France
France (**) Orange Labs Networks Lannion, 2 avenue Pierre Marzin, 22307 Lannion Cedex, Lannion, France}
\email{[ridha.nasri, alain.simonian, fabrice.guillemin]@orange.com}

\begin{document}

\date{Version of \today}

\begin{abstract}
For any complex parameters $x$ and $\nu$, we provide a new class of linear inversion formulas 
$T = A(x,\nu) \cdot S \Leftrightarrow S = B(x,\nu) \cdot T$ 
between sequences $S = (S_n)_{n \in \mathbb{N}^*}$ and 
$T = (T_n)_{n \in \mathbb{N}^*}$, where the infinite lower-triangular matrix 
$A(x,\nu)$ and its inverse $B(x,\nu)$ involve Hypergeometric polynomials 
$F(\cdot)$, namely
$$
	\left\{
	\begin{array}{ll}
	A_{n,k}(x,\nu) = \displaystyle (-1)^k\binom{n}{k}F(k-n,-n\nu;-n;x),
  \\ \\
	B_{n,k}(x,\nu) = \displaystyle (-1)^k\binom{n}{k}F(k-n,k\nu;k;x)
	\end{array} \right.
$$
for $1 \leqslant k \leqslant n$. Functional relations between the ordinary 
(resp. exponential) generating functions of the related sequences $S$ and $T$ are also given. 

These new inversion formulas have been initially motivated by the resolution of an integral equation recently appeared in the field of Queuing Theory; we apply them to the full resolution of this integral equation. Finally, matrices involving generalized Laguerre polynomials polynomials are discussed as specific cases of our general inversion scheme.
\end{abstract}

\maketitle 


\section{Introduction}


In this Introduction, we present a general class of linear inversion formulas with coefficients involving Hypergeometric polynomials and motivate the need for such formulas.  After an overview of the recent state-of-the-art in the corresponding field, we summarize the main contributions of this paper.

\subsection{Motivation}
\label{IM}
The need for an inversion formula whose coefficients involve Hypergeometric polynomials is motivated by the resolution of an integral equation arising  from Queuing Theory \cite{GQSN18}, 
which can be formulated as follows:


\textit{\textbf{given a constant $U^- > 0$,  a real function $\mathfrak{R}$ on $[0,U^-]$ (with $\mathfrak{R}(U^-) = 0$) and an entire function $A$ in $\mathbb{C}$, solve the integral equation}}
\begin{equation}
\int_0^{U^-} E^*(\zeta \, \mathfrak{R}(\zeta) \cdot z) \, 
e^{- \mathfrak{R}(\zeta) \cdot z} \, \mathrm{d}\zeta = A(z), 
\qquad z \in \mathbb{C},
\label{EI0}
\end{equation}
\textit{\textbf{for an unknown entire function $E^*$ in $\mathbb{C}$ with 
$E^*(0) = 0$}.}

The product $\zeta \, \mathfrak{R}(\zeta)$ intervening in the argument of 
$E^*$ in (\ref{EI0}) being not one-to-one on interval $[0,U^-]$ (it vanishes at both $\zeta = 0$ and $\zeta = U^-$), this integral equation is not amenable to a standard Fredholm equation of the first kind (\cite{PolMan98}, Chap.3, 
3.1.6). An exponential power series
\begin{equation}
E^*(z) = \sum_{\ell = 1}^{+\infty} E_\ell \, \frac{z^\ell}{\ell!}, 
\qquad z \in \mathbb{C},
\label{defE*}
\end{equation}
for an entire solution $E^*$, however, drives the resolution of (\ref{EI0}) to that of the infinite lower-triangular linear system
\begin{equation}
\forall \, b \in \mathbb{N}^*, \qquad 
\sum_{\ell = 1}^b (-1)^\ell \binom{b}{\ell} \, 
Q_{b,\ell} \, E_{\ell} = K_b,  
\label{T0}
\end{equation}
with unknown $E_\ell$, $\ell \in \mathbb{N}^*$, and coefficient matrix 
$Q = (Q_{b,\ell})_{b, \ell \in \mathbb{N}^*}$ given by
\begin{equation}
Q_{b,\ell} = - \frac{\Gamma(b)\Gamma(1-b\nu)}{\Gamma(b-b\nu)} \, 
(U^-)^{\ell+1} \, \frac{x^{1-b}}{1-x} \; F(\ell - b,-b \nu;-b;x), 
\qquad 1 \leqslant \ell \leqslant b.
\label{Q0}
\end{equation}
In (\ref{Q0}), $\Gamma$ is the Euler Gamma function and 
$F(\alpha,\beta;\gamma;\cdot)$ denotes the Gauss Hypergeometric function with complex parameters $\alpha$, $\beta$, $\gamma \notin -\mathbb{N}$; besides, 
$U^- > 0$, $x$ and $\nu < 0$ are known real parameters (whose specification is not needed). Recall that $F(\alpha,\beta;\gamma;\cdot)$ reduces to a polynomial with degree $-\alpha$ (resp. $-\beta$) if $\alpha$ 
(resp. $\beta$) equals a non positive integer; expression (\ref{Q0}) for coefficient $Q_{b,\ell}$ thus involves a Hypergeometric polynomial with degree $b - \ell$ in both arguments $x$ and $\nu$. At this stage, the explicit expression of the right-hand side $K_b$ in (\ref{T0}) is not necessary. 

Diagonal coefficients $Q_{b,b}$, $b \geqslant 1$, are non-zero so that lower-triangular system (\ref{T0}) has a unique solution; equivalently, this proves the uniqueness of the  entire solution $E^*$ to (\ref{EI0}) with power series expansion (\ref{defE*}). This solution, nevertheless, needs to be made explicit in terms of parameters; to this end, write system (\ref{T0}) equivalently as 
\begin{equation}
\forall \, b \in \mathbb{N}^*, \qquad 
\sum_{\ell = 1}^b A_{b,\ell}(x,\nu) \, \widetilde{E}_\ell = \widetilde{K}_b,
\label{T0BIS}
\end{equation}
with the reduced unknowns and right-hand side
$$
\widetilde{E}_\ell = ( U^- )^{\ell+1} \cdot E_\ell, 
\qquad
\widetilde{K}_b = - \frac{\Gamma(b-b\nu)}{\Gamma(b)\Gamma(1-b\nu)} (1-x)x^{b-1}\cdot K_b,
$$ 
and coefficients
\begin{equation}
A_{b,\ell}(x,\nu) = (-1)^\ell \binom{b}{\ell} F(\ell-b,-b\nu;-b;x), 
\qquad 1 \leqslant \ell \leqslant b.
\label{T0TER}
\end{equation}
As shown in the present paper, it proves that that the linear relation 
(\ref{T0BIS}) to which initial system (\ref{T0}) has been recast is always amenable to an explicit inversion for any right-hand side 
$(K_b)_{b \in \mathbb{N}^*}$, the inverse matrix 
$B(x,\nu) = A(x,\nu)^{-1}$ involving also Hypergeometric polynomials. This consequently solves system (\ref{T0}) explicitly, hence integral equation 
(\ref{EI0}).

Beside the initial motivation stemming from an integral equation, the remarkable structure of the inversion scheme $B(x,\nu) = A(x,\nu)^{-1}$ obtained in this paper brings a new contribution to the realm of linear inversion formulas, namely infinite lower-triangular matrices with coefficients involving Hypergeometric polynomials; as shown in the following, other polynomial families can also be included in this pattern. In the following sub-section, we position the originality of the present contribution with respect to known inversion patterns. 

\subsection{State-of-the-art}
We here review the known classes of linear inversion formulas provided by the recent literature, most of them motivated by problems from pure Combinatorics together with the determination of remarkable relations on special functions.  Given a complex sequence $(a_j)_{j \in \mathbb{N}}$, it has been early shown \cite{GouldHsu73} that the lower triangular matrices $A$ and 
$B$ with coefficients
$$
A_{n,k} = \frac{1}{(n-k)!} \prod_{j=k}^{n-1}(a_j + k), 
\quad
B_{n,k} = \frac{a_k+k}{a_n+n} \cdot \frac{(-1)^{n-k}}{(n-k)!} 
\prod_{j=k+1}^{n}(a_j + n)
$$
for $k \leqslant n$ (with a product over an empty set being set 1) are inverses. These inversion formulas actually prove to be a particular case of the general Krattenthaler formulas \cite{Kratten96} stating that, given complex sequences $(a_j)_{j \in \mathbb{Z}}$, $(b_j)_{j \in \mathbb{Z}}$ and 
$(c_j)_{j \in \mathbb{Z}}$ with $c_j \neq c_k$ for $j \neq k$, the lower triangular matrices $A$ and $B$ with coefficients
\begin{equation}
A_{n,k} = 
\frac{\displaystyle \prod_{j=k}^{n-1}(a_j + b_j c_k)}
{\displaystyle \prod_{j=k+1}^{n}(c_j - c_k)} , 
\qquad
B_{n,k} = \frac{a_k + b_kc_k}{a_n + b_nc_n} \cdot 
\frac{\displaystyle \prod_{j=k+1}^{n}(a_j + b_j c_n)}
{\displaystyle \prod_{j=k}^{n-1}(c_j - c_n)}
\label{Kratt0}
\end{equation}
for $k \leqslant n$, are inverses; the proof of (\ref{Kratt0}) relies on the existence of linear operators $\mathscr{U}$, $\mathscr{V}$ on the linear space of formal Laurent series such that 
$$
\mathscr{U} f_k(z) = c_k \cdot \mathscr{V} f_k(z), \qquad k \in \mathbb{Z},
$$
where $f_k(z) = \sum_{n \geqslant k} A_{n,k} z^n$; the partial Laurent series 
$g_n(z) = \sum_{k \leqslant n} B_{n,k} z^{-k}$, $n \in \mathbb{Z}$, for the inverse inverse $B = A^{-1}$ can then be expressed in terms of the adjoint operator $\mathscr{V}^*$ of $\mathscr{V}$. A generalization of inverse relation 
(\ref{Kratt0}) to the multi-dimensional case when 
$A = (A_{\mathbf{n},\mathbf{k}})$ with indexes $\mathbf{n}$, 
$\mathbf{k} \in \mathbb{Z}^r$ for some integer $r$ has also been provided in \cite{Schlo97}; as an application, the obtained relations bring summation formulas for multidimensional basic hypergeometric series. 

The lower triangular matrix $A = A(x,\nu)$ introduced in 
(\ref{T0BIS})-(\ref{T0TER}), however, cannot be cast into the specific product form (\ref{Kratt0}) for its inversion: in fact, such a product form for the coefficients of $A(x,\nu)$ should involve the $n-k$ zeros $c_{j,n,k}$, 
$k \leqslant j \leqslant n-1$ of the Hypergeometric polynomial 
$F(k-n,-n\nu;-n;x)$, $k \leqslant n$, in variable $x$; but such zeros depend on all indexes $j$, $n$ and $k$, which precludes the use of a factorization such as (\ref{Kratt0}) where sequences with one index only intervene. In this paper, using functional operations on specific generating series related to its coefficients, we will show how matrix $A(x,\nu)$ can be nevertheless  inverted through a fully explicit procedure.

\subsection{Paper contribution}
Our main contributions can be summarized as follows:

$\bullet$ in Section \ref{LTS}, we first establish an inversion criterion for a class of infinite lower-triangular matrices, which enables us to state the inversion formula for the considered class of lower triangular matrices with Hypergeometric polynomials;

$\bullet$ in Section \ref{GF}, functional relations are obtained for ordinary 
(resp. exponential) generating functions of sequences related by the inversion formula;

$\bullet$ applying the latter general results, the infinite linear system 
(\ref{T0BIS}) motivated above is fully solved; both the ordinary and exponential generating functions associated with its solution are, in particular, given an integral representation (Section \ref{SecA1}). Finally, matrices depending on other families of special polynomials \---- namely, generalized Laguerre polynomials, are discussed as specific cases of our general inversion scheme (Section \ref{SecA2}).


\section{Lower-Triangular Systems}
\label{LTS}


Let $(a_m)_{m\in \N}$ and $(b_m)_{m\in \N}$ be complex sequences such that 
$a_0 = b_0 =1$ and denote by $f(x)$ and $g(x)$ their respective exponential generating series, i.e.,
\begin{equation}
f(x)=\sum_{m=0}^{+\infty}\frac{a_m}{m!} \, x^m, 
\qquad 
g(x)=\sum_{m=0}^{+\infty}\frac{b_m}{m!} \, x^m; 
\label{DefFG}
\end{equation}
in the following, we will use the notation $[x^n] f(x)$ for the coefficient of 
$x^n$, $n \in \mathbb{N}$, in power series $f(x)$. For all $x \in \mathbb{C}$, define the infinite lower-triangular matrices 
$A(x) = (A_{n,k}(x))_{n,k \in \N^*}$ and 
$B(x) = (B_{n,k}(x))_{n,k \in \N^*}$ by
\begin{equation}
	\left\{
	\begin{array}{ll}
	A_{n,k}(x) = \displaystyle 
	(-1)^k \binom{n}{k}\sum_{m=0}^{n-k}\frac{(k-n)_m \, a_m}{m!} \, x^m,
	\\ \\
	B_{n,k}(x) = \displaystyle 
	(-1)^k \binom{n}{k}\sum_{m=0}^{n-k}\frac{(k-n)_m \, b_m}{m!} \, x^m
	\end{array} \right.
\label{DefAB}
\end{equation}
for $1 \leqslant k \leqslant n$ ($(c)_m$, $c \in \mathbb{C}$, 
$m \in \mathbb{N}^*$, denotes the Pochhammer symbol (\cite{NIST10}, §5.2(iii)) with $(c)_0 = 1$). From definition (\ref{DefAB}), matrices $A(x)$ and $B(x)$ have diagonal elements equal to $A_{k,k}(x) = B_{k,k}(x) = (-1)^k$, 
$k \in \mathbb{N}^*$, and are thus invertible. 

\subsection{An inversion criterion}
We first state the following inversion criterion.

\begin{prop}
\textbf{Matrices $A(x)$ and $B(x)$ are inverse of each other if and only if the condition}
	\begin{equation}
	[x^{n-k}]f(-x)g(x)=\delta(n-k), 
	\qquad 1 \leqslant k \leqslant n,
	\label{Invers0}
	\end{equation}
\textbf{on functions $f$ and $g$ holds.}
\label{theoMaininversionR}
\end{prop}

\noindent
The proof of Proposition \ref{theoMaininversionR} requires the following technical lemma whose proof is deferred to Appendix \ref{A1}.

\begin{lemma}
	\textbf{Let $N \in \N^*$ and complex numbers $\lambda$, $\mu$. Defining}
	$$
	D_{N}(\lambda,\mu) = 
	\sum_{r=0}^{N-1}\frac{(-1)^r}{\Gamma(1+r-\lambda)\Gamma(1-r+\mu)},
	$$
	\textbf{we have}
\begin{equation}
	D_{N}(\lambda,\mu) = 
	\left \{
	\begin{array}{ll}
	\displaystyle 
	\frac{1}{\mu - \lambda} \left[ \frac{1}{\Gamma(-\lambda)\Gamma(1+\mu)} - 
	\frac{(-1)^N}{\Gamma(N-\lambda)\Gamma(1-N+\mu)} \right], \, \mu \neq \lambda
	\\ \\
	\displaystyle \frac{\sin(\pi\lambda)}{\pi} \left[ \psi(-\lambda)-\psi(N-\lambda) \right], 
	~ \qquad \qquad \qquad \qquad \quad \, \mu = \lambda,
	\end{array} \right.
\label{Sn}
\end{equation}
\textbf{where $\psi$ denotes the logarithmic derivative $\Gamma'/\Gamma$.}
	\label{lemm1}
\end{lemma}

\noindent
We now proceed with the justification of Proposition \ref{theoMaininversionR}.

\begin{proof}
$A(x)$ and $B(x)$ being lower-triangular, so is their product 
$C(x) = A(x)B(x)$. After definition (\ref{DefAB}), 
the coefficient 
$C_{n,k}(x) = \sum_{\ell \geqslant 1} A_{n,\ell}(x)B_{\ell,k}(x)$, 
$1 \leqslant k \leqslant n$ (where the latter sum over index $\ell$ is actually finite), of matrix $C(x)$ reads 
\begin{align}
C_{n,k}(x) = & \, 
\sum_{\ell = 1}^{+\infty} (-1)^\ell \, \frac{ n!}{\ell!(n-\ell)!} 
\sum_{m=0}^{n-\ell} \frac{(-1)^m(n-\ell)! \, a_m}{(n-\ell-m)!m!} \, x^m \; \times  
\nonumber \\
& \, (-1)^k \, \frac{\ell!}{k!(\ell-k)!}\sum_{m'=0}^{\ell-k} 
\frac{(-1)^{m'}(\ell-k)! \, b_{m'}}{(\ell-k-m')m'!} \, x^{m'}
\nonumber
\end{align}
after writing $(-r)_m = (-1)^m r!/(r-m)!$ for any positive integer $r$, that is,
\begin{equation}
C_{n,k}(x) = (-1)^k \frac{n!}{k!} \sum_{\ell = 1}^{+\infty} (-1)^\ell 
\sum_{m = 0}^{n-\ell} \frac{(-1)^m a_m \, x^m}{m!(n-\ell-m)!} 
\sum_{m' = 0}^{\ell - k} \frac{(-1)^{m'} b_{m'} \, x^{m'}}{m'!(\ell-k-m')!}.
\label{P11}
\end{equation}
Exchanging the summation order in (\ref{P11}) further gives
\begin{align}
C_{n,k}(x) = (-1)^k \frac{n!}{k!} & \, 
\sum_{(m,m') \in \Delta_{n,k}} \frac{(-1)^m a_m \, x^m}{m!}
\frac{(-1)^{m'} b_{m'} \, x^{m'}}{m'!} \; \times
\nonumber \\
& \sum_{k \leqslant \ell \leqslant n}\frac{(-1)^\ell}{(n-\ell-m)!(\ell-k-m')!}
\label{P12}
\end{align}
with $\Delta_{n,k} = \{(m,m') \in \mathbb{N}^2, \; m+m' \leqslant n-k\}$ and where the latter summation on index $\ell$ equivalently reads
\begin{align}
\sum_{k \leqslant \ell \leqslant n}\frac{(-1)^\ell}{(n-\ell-m)!(\ell-k-m')!} 
= & \, \sum_{r = 0}^{n-k} \frac{(-1)^{n-r}}{(r-m)!(n-r-k-m')!} 
\nonumber \\
= & \, (-1)^n \, D_{n-k+1}(m,n-k-m')
\nonumber
\end{align}
with the index change $\ell = n-r$ and the notation of Lemma \ref{lemm1}. The expression (\ref{P12}) for coefficient $C_{n,k}(x)$ consequently reduces to
\begin{align}
C_{n,k}(x) = (-1)^{n+k} \, \frac{n!}{k!} \, 
\sum_{(m,m') \in \Delta_{n,k}} & \frac{(-1)^m a_m \, x^m}{m!}
\frac{(-1)^{m'} b_{m'} \, x^{m'}}{m'!} \; \times
\nonumber \\
& D_{n-k+1}(m,n-k-m')
\label{P13}
\end{align}
and we are left to calculate $D_{n-k+1}(m,n-k-m')$ for all non negative 
$m$ and $m'$. By Lemma \ref{lemm1} applied to $\lambda = m$ and 
$\mu = n-k-m'$, we successively derive that

\begin{itemize}
\item[\textbf{(a)}] if $\mu > \lambda \Leftrightarrow m + m' < n-k$, formula 
(\ref{Sn}) entails
\begin{align}
& D_{n-k+1}(m,n-k-m') \; = 
\nonumber \\ 
& \frac{1}{n-k-(m+m')} \left [ \frac{1}{\Gamma(-m)\Gamma(1+n-k-m')} - 
\frac{(-1)^{n-k+1}}{\Gamma(n-k+1-m)\Gamma(-m')} \right ];
\nonumber
\end{align}
as $\Gamma(-m) = \Gamma(-m') = \infty$ for all non negative integers 
$m \geqslant 0$ and $m' \geqslant 0$, each fraction of the latter expression vanishes and thus
\begin{equation}
D_{n-k+1}(m,n-k-m') = 0, \qquad m + m' < n-k;
\label{P14}
\end{equation}

\item[\textbf{(b)}] if $\lambda = \mu \Leftrightarrow m + m' = n-k$, formula 
(\ref{Sn}) entails
\begin{equation}
D_{n-k+1}(m,m) = \lim_{\lambda \rightarrow m} \frac{\sin(\pi\lambda)}{\pi} 
\left [ \psi(-\lambda) - \psi(n-k+1-\lambda) \right ].
\label{P15}
\end{equation}
We have $\sin(m \pi) = 0$ while function $\psi$ has a polar singularity at every non positive integer; the limit (\ref{P15}) is therefore indeterminate 
($0 \times \infty$) but this is solved by invoking the reflection formula 
$\psi(z) - \psi(1-z) = - \pi \, \cot(\pi \, z)$, $z \notin - \mathbb{N}$, 
for function $\psi$ (\cite{NIST10}, Chap.5, §5.5.4). In fact, applying the latter to $z = -\lambda$ first gives
$\sin(\pi \lambda) \, \psi(-\lambda) = 
\sin(\pi \lambda) \, \psi(1+\lambda) + \pi \cdot \cos(\pi\lambda)$ whence
$$
\lim_{\lambda \rightarrow m} \frac{\sin(\pi \lambda)}{\pi} \, \psi(-\lambda) 
= 0 \times \psi(1+m) + (-1)^m = (-1)^m;
$$
besides, the second term $\psi(n-k+1-\lambda)$ in (\ref{P15}) has a finite limit when $\lambda \rightarrow m$ since $m + m' = n - k \Rightarrow m \leqslant n - k$ so that $n-k+1-\lambda$ tends to a positive integer. From 
(\ref{P15}) and the latter discussion, we are left with
\begin{equation}
D_{n-k+1}(m,m) = (-1)^m, \qquad m + m' = n-k.
\label{P16}
\end{equation}
\end{itemize}

In view of the previous items \textbf{(a)} and \textbf{(b)}, identities 
(\ref{P15}) and (\ref{P16}) together reduce expression (\ref{P13}) to
\begin{align}
C_{n,k}(x) = & \; (-1)^{n+k} \, \frac{n!}{k!} \, 
\sum_{m = 0}^{n-k} \frac{(-1)^m a_m \, x^m}{m!} \, (-1)^{n-k-m} 
\frac{b_{n-k-m} \, x^{n-k-m}}{(n-k-m)!} \times (-1)^m 
\nonumber \\
= & \; \frac{n!}{k!} \, \sum_{m = 0}^{n-k} \frac{(-1)^m a_m \, x^m}{m!} \, 
\frac{b_{n-k-m}}{(n-k-m)!}x^{n-k} = \frac{n!}{k!} \; [x]^{n-k}f(-x)g(x)
\nonumber
\end{align}
where $f$ and $g$ denote the exponential generating function of the sequence 
$(a_m)_{m \in \mathbb{N}^*}$ and the sequence $(b_m)_{m \in \mathbb{N}^*}$, respectively. It follows that $C(x) = A(x) B(x)$ is the identity matrix $\mathrm{Id}$ if and only if condition (\ref{Invers0}) holds, as claimed. 
\end{proof}

Following the proof of Proposition \ref{theoMaininversionR}, the same arguments apply to the general case when the sequences 
$(a_m)_{m \in \mathbb{N}}$ and $(b_m)_{m \in \mathbb{N}}$ associated with 
lower-triangular matrices $A(x)$ and $B(x)$ are also given for each pair of indexes $n, k$, that is, 
\begin{equation}
	\left\{
	\begin{array}{ll}
	A_{n,k}(x) = \displaystyle 
	(-1)^k \binom{n}{k}\sum_{m=0}^{n-k}\frac{(k-n)_m \, a_{m;n,k}}{m!} \, x^m,
	\\ \\
	B_{n,k}(x) = \displaystyle 
	(-1)^k \binom{n}{k}\sum_{m=0}^{n-k}\frac{(k-n)_m \, b_{m;n,k}}{m!} \, x^m
	\end{array} \right.
\label{DefABbis}
\end{equation}
for $1 \leqslant k \leqslant n$. Condition (\ref{Invers0}) for 
$A(x) B(x) = \mathrm{Id}$ then simply extends to
\begin{equation}
	[x^{n-k}]f_{n,k}(-x)g_{n,k}(x)=\delta(n-k), 
	\qquad 1 \leqslant k \leqslant n,
	\label{Invers0bis}
	\end{equation}
where $f_{n,k}$ (resp. $g_{n,k}$) denotes the exponential generating function of the sequence $(a_{m;n,k})_{m \in \mathbb{N}}$ (resp.  
$(b_{m;n,k)})_{m \in \mathbb{N}}$) for given $n, \, k \in \mathbb{N}^*$. This straightforward generalization of Proposition \ref{theoMaininversionR} will be hereafter invoked to verify the inversion criterion.

\subsection{The inversion formula}
We now formulate the inversion formula for lower-triangular matrices involving Hypergeometric polynomials.

\begin{theorem}
\textbf{Let $x, \nu \in \mathbb{C}$ and define the lower-triangular matrices 
$A(x,\nu)$ and $B(x,\nu)$ by}
	\begin{equation}
	\left\{
	\begin{array}{ll}
	A_{n,k}(x,\nu) = \displaystyle (-1)^k\binom{n}{k}F(k-n,-n\nu;-n;x),
  \\ \\
	B_{n,k}(x,\nu) = \displaystyle (-1)^k\binom{n}{k}F(k-n,k\nu;k;x)
	\end{array} \right.
	\label{DefABxNU}
	\end{equation}
\textbf{for $1 \leqslant k \leqslant n$. For any pair of complex sequences 
$(S_n)_{n \in \mathbb{N}^*}$ and $(T_n)_{n \in \mathbb{N}^*}$, the inversion formula}
	\begin{equation}
	T_n = \sum_{k=1}^{n}A_{n,k}(x,\nu)S_k 
	\Longleftrightarrow 
	S_n=\sum_{k=1}^{n}B_{n,k}(x,\nu)T_k, \quad n \in \mathbb{N}^*,
	\label{eq:inversionR}
	\end{equation}
\textbf{holds.}
	\label{PropIn}
\end{theorem}

\noindent
\begin{remark}
\textbf{a)} Note that the factor $F(k-n,-n\nu;-n;x)$ in the definition 
(\ref{DefABxNU}) of matrix $A(x,\nu)$ is always well-defined although the third argument $-n$ is a negative integer; in fact, given 
$1 \leqslant k \leqslant n$, write by definition (\cite{NIST10}, 15.2.1) 
\begin{equation}
F(k-n,-n\nu;-n;x) = 
\sum_{m=0}^{n-k} \frac{(k-n)_m(-n\nu)_m}{(-n)_m \, m!} x^m
\label{DefFpoly}
\end{equation}
and the denominator $(-n)_m = (-1)^m n!/(n-m)!$ therefore never vanishes for 
all indexes $m \leqslant n-k < n$;

\textbf{b)} the polynomial factors $F(k-n,-n\nu;-n;x)$ and $F(k-n,k\nu;k;x)$ respectively intervening in coefficients $A_{n,k}(x,\nu)$ and $B_{n,k}(x,\nu)$ in definition (\ref{DefABxNU}) are deduced from each other by the substitution 
$k \leftrightarrow -n$. This simple substitution, however, does not leave the remaining factor $\binom{n}{k}$ invariant and thus cannot carry out by itself the inversion scheme (\ref{eq:inversionR}). 
\end{remark}

\begin{proof}
To show that $A(x,\nu)B(x,\nu) = \mathrm{Id}$, it is sufficient to verify  criterion (\ref{Invers0bis}). From (\ref{DefABbis}), we first specify the associated sequences $(a_{m;n,k})_{m \in \mathbb{N}}$ and 
$(b_{m;n,k})_{m \in \mathbb{N}}$ for a given pair $(n,k)$. On one hand, 
(\ref{DefFpoly}) entails $a_{m;n} = (-n\nu)_m/(-n)_m$, $m \geqslant 0$, 
for given $n \in \mathbb{N}^*$ and, in particular, $a_{0;n} = 1$; similarly, write
\begin{equation}
F(k-n,k\nu,k;x) = \sum_{m=0}^{n-k} \frac{(k-n)_m(k\nu)_m}{(k)_m \, m!} \, x^m 
\label{DefFFbis}
\end{equation}
so that
$b_{m;k} = (k\nu)_m/(k)_m$, $m \geqslant 0$, for given $k \in \mathbb{N}^*$ with $b_{0;k} = 1$. Let $f_n$ and $g_{k}$ respectively denote the exponential generating function of these sequences $(a_{m;n})_{m \geqslant 0}$ and 
$(b_{m;k})_{m \geqslant 0}$; the product $f_n(-x)g_{k}(x)$ is then given by
\begin{align}
f_n(-x)g_{k}(x) & \, = \left ( \sum_{m \geqslant 0} (-1)^m 
\frac{a_{m;n}}{m!} \, x^m \right ) 
\left ( \sum_{m \geqslant 0} \frac{b_{m;k}}{m!} \, x^m \right )
\nonumber \\
& \, = 
\sum_{m=0}^{+\infty} (-1)^m \frac{(-n\nu)_m}{(-n)_m \, m!} \, x^m \cdot 
\sum_{m=0}^{+\infty} \frac{(k\nu)_m}{(k)_m \, m!} \, x^m = 
\sum_{\ell \geqslant 0} U_\ell^{(n,k)} \, x^\ell
\nonumber
\end{align}
where
\begin{equation}
U_\ell^{(n,k)} = \sum_{m=0}^\ell (-1)^m \frac{(-n\nu)_m}{(-n)_m \, m!} 
\frac{(k\nu)_{\ell-m}}{(k)_{\ell-m} \, (\ell-m)!}, \qquad \ell \geqslant 0.
\label{DefU}
\end{equation}
Let then $n \geqslant k$; from expression (\ref{DefU}), we derive
\begin{align}
U_{n-k}^{(n,k)} = & \, \sum_{m=0}^{n-k} (-1)^m \frac{(-n\nu)_m}{(-n)_m \, m!} \cdot 
\frac{(k\nu)_{n-k-m}}{(k)_{n-k-m} \, (n-k-m)!}
\nonumber \\
= & \, \sum_{m=0}^{n-k} (-1)^m \frac{\Gamma(m-n\nu)}{\Gamma(-n\nu)} 
\cdot \frac{(-1)^m(n-m)!}{n!} \cdot \frac{1}{m!} \cdot 
\frac{\Gamma(n-k-m+k\nu)}{\Gamma(k\nu)} \; \times 
\nonumber \\
& \qquad \; \; \; 
\frac{\Gamma(k)}{\Gamma(n-k-m+k)} \cdot \frac{1}{(n-k-m)!}
\nonumber
\end{align}
after writing the Pochhammer symbol $(c)_m = \Gamma(m+c)/\Gamma(c)$ for 
$c \notin -\mathbb{N}$ and noting that $(-n)_m = (-1)^m n!/(n-m)!$. Reducing the latter expression of $U_{n-k}^{(n,k)}$ gives
\begin{align}
U_{n-k}^{(n,k)} = & \, \frac{\Gamma(k)}{n!\Gamma(-n\nu)\Gamma(k\nu)} 
\sum_{m=0}^{n-k} (n-m) \frac{\Gamma(m-n\nu)\Gamma(n-k-m+k\nu)}{m!(n-k-m)!}
\nonumber \\
= & \, \frac{\Gamma(k)}{n!\Gamma(-n\nu)\Gamma(k\nu)} 
(X_{n-k}^{(n,k)} + Y_{n-k}^{(n,k)})
\label{DefUBIS}
\end{align}
where we introduce the sums (after decomposing $n-m = k + (n-m-k)$)
$$
\left\{
\begin{array}{ll}
X_{n-k}^{(n,k)} = \displaystyle k \cdot \sum_{m=0}^{n-k} 
\frac{\Gamma(m-n\nu)\Gamma(n-k-m+k\nu)}{m!(n-k-m)!},
\\ \\
Y_{n-k}^{(n,k)} = \displaystyle \sum_{m=0}^{n-k} 
(n-m-k) \cdot \frac{\Gamma(m-n\nu)\Gamma(n-k-m+k\nu)}{m!(n-k-m)!}.
\end{array} \right.
$$
To calculate $X_{n-k}^{(n,k)}/k$, note that this equals to the coefficient of 
$x^{n-k}$ in the power series expansion of the product 
\begin{align}
& \left ( \sum_{m=0}^{+\infty} \frac{\Gamma(m-n\nu)}{m!} \, x^m \right ) 
\left ( \sum_{m=0}^{+\infty} \frac{\Gamma(m+k\nu)}{m!} \, x^m \right ) \; = 
\nonumber \\
& \left ( \sum_{m=0}^{+\infty} \frac{\Gamma(-n\nu)(-n\nu)_m}{m!} \, x^m \right )
\left ( \sum_{m=0}^{+\infty} \frac{\Gamma(k\nu)(k\nu)_m}{m!} \, x^m \right ) = 
\frac{\Gamma(-n\nu)}{(1-x)^{-n\nu}} \cdot \frac{\Gamma(k\nu)}{(1-x)^{k\nu}}
\nonumber
\end{align}
so that
\begin{equation}
X_{n-k}^{(n,k)} = k \, \Gamma(-n\nu)\Gamma(k\nu) \cdot 
[x]^{n-k} \left \{ \frac{(1-x)^{n\nu}}{(1-x)^{k\nu}} \right \}.
\label{DefUTER}
\end{equation}
As to the sum $Y_{n-k}^{(n,k)}$, it equals the coefficient of $x^{n-k}$ in the power series expansion of the product 
$$
\left ( \sum_{m=0}^{+\infty} \frac{\Gamma(m-n\nu)}{m!} \, x^m \right ) \cdot 
x \, \frac{\mathrm{d}}{\mathrm{d}x} 
\left [ \frac{\Gamma(k\nu)}{(1-x)^{k\nu}} \right ] = 
\frac{\Gamma(-n\nu)}{(1-x)^{-n\nu}} \times x \, 
\Gamma(k\nu) \, \frac{k\nu}{(1-x)^{k\nu+1}}
$$
so that
\begin{equation}
Y_{n-k}^{(n,k)} = \Gamma(-n\nu)\Gamma(k\nu+1) \cdot [x]^{n-k} 
\left \{ \frac{x(1-x)^{n\nu}}{(1-x)^{k\nu+1}} \right \}.
\label{DefUQUATER}
\end{equation}
Using formulas (\ref{DefUTER}) and (\ref{DefUQUATER}) for sums 
$X_{n-k}^{(n,k)}$ and $Y_{n-k}^{(n,k)}$, the expression (\ref{DefUBIS}) for 
$U_{n-k}^{(n,k)}$ then easily reduces to
\begin{align}
U_{n-k}^{(n,k)} = & \, \frac{[x]^{n-k}}{n!} \left \{ \Gamma(k+1) 
\frac{(1-x)^{n\nu}}{(1-x)^{k\nu}} + k \nu \, \Gamma(k) 
\frac{x(1-x)^{n\nu}}{(1-x)^{k\nu+1}} \right \}
\nonumber \\
= & \, \frac{k!}{n!} \, 
\Bigl \{ [x^{n-k}] (1-x)^{(n-k)\nu-1} (1+(\nu-1)x) \Bigr \}, 
\qquad n \geqslant k.
\label{DefU5}
\end{align}
With the series expansion $(1-x)^{(n-k)\nu-1} = \sum_{\ell \geqslant 0} x^\ell 
(1-(n-k)\nu)_\ell/\ell!$, expression (\ref{DefU5}) for $n-k \geqslant 1$ then gives
\begin{align}
U_{n-k}^{(n,k)} = & \frac{k!}{n!} \left \{ \frac{(1-(n-k)\nu)_{n-k}}{(n-k)!} + 
(\nu - 1) \frac{(1-(n-k)\nu)_{n-k-1}}{(n-k-1)!} \right \} 
\nonumber \\
= & \frac{k!}{n!} \frac{1}{\Gamma(1-(n-k)\nu)} \frac{1}{(n-k)!} 
\Bigl \{ \Gamma(1-(n-k)\nu +n-k) \; + 
\nonumber \\
& \qquad \qquad \qquad \qquad \qquad \qquad 
(\nu-1)\Gamma((n-k)(1-\nu))(n-k) \Bigr \}
\nonumber
\end{align}
by definition of the Pochhammer symbol, and the relation 
$\Gamma(1+z) = z \Gamma(z)$ applied to the argument 
$z = (n-k)\nu +n-k = (n-k)(1-\nu)$ entails 
\begin{align}
U_{n-k}^{(n,k)} = & 
\frac{k!}{n!} \frac{1}{\Gamma(1-(n-k)\nu)} \frac{1}{(n-k)!} 
\Bigl \{ (n-k)(1-\nu)\Gamma((n-k)(1-\nu)) \; + 
\nonumber \\
& \qquad \qquad \qquad \qquad \qquad \qquad 
(\nu-1)\Gamma((n-k)(1-\nu))(n-k) \Bigr \}
\nonumber
\end{align}
so that $U_{n-k}^{(n,k)} = 0$ for $n - k \geqslant 1$. Now if $n = k$, 
(\ref{DefU5}) reduces to
$$
U_{n-k}^{(n,k)} = [x^0] \left \{ 1 + \frac{\nu \, x}{1-x} \right \} = 1.
$$
The inversion condition (\ref{Invers0bis}) for 
$U_{n-k}^{(n,k)} = [x]^{n-k}f_n(-x)g_{k}(x)$ is therefore fulfilled for all $n, \, k \geqslant 1$ and we conclude that inverse relation 
(\ref{eq:inversionR}) holds for any pair of sequences $(S_n)_{n \geqslant 1}$ and $(T_n)_{n \geqslant 1}$.
\end{proof}


\section{Generating functions}
\label{GF}

As a direct consequence of Theorem \ref{PropIn}, remarkable functional relations can be derived for the ordinary (resp. exponential) generating functions of sequences related by the inversion formula. We first address ordinary generating functions and state the following reciprocal relations.

\begin{corol}
\textbf{For given complex parameters $x$ and $\nu$, let 
$(S_n)_{n \in \mathbb{N}^*}$ and $(T_n)_{n \in \mathbb{N}^*}$ be sequences related by the inversion formulas (\ref{eq:inversionR}) of Theorem 
\ref{PropIn}, that is, 
$S = B(x,\nu) \cdot T \Leftrightarrow T = A(x,\nu) \cdot S$.}

\textbf{Denote by $\mathfrak{G}_S(z)$ and $\mathfrak{G}_T(z)$ the formal ordinary generating series of $S$ and $T$, respectively. Defining the mapping 
$\Xi$ (depending on parameters $x$ and $\nu$) by}
\begin{equation}
\Xi(z) = \frac{z}{z-1}\Bigl(\frac{1-z}{1-z(1-x)}\Bigr)^{\nu},
\label{DefXi}
\end{equation}
\textbf{the relation}
\begin{equation}
\mathfrak{G}_S(z) = \left [ \frac{1-\nu}{1-z} + \frac{\nu}{1-z(1-x)} \right ]
\mathfrak{G}_T(\Xi(z))
\label{eq:OGFRelation}
\end{equation}
\textbf{holds. Conversely, $\mathfrak{G}_T$ is given in terms of 
$\mathfrak{G}_S$ by}
\begin{equation}
\mathfrak{G}_T(\xi) = \mathfrak{G}_S(\Omega(\xi)) 
\left [ \frac{1-\nu}{1-\Omega(\xi)} + 
\frac{\nu}{1-(1-x)\Omega(\xi)} \right ]^{-1}
\label{eq:OGFRelationBIS}
\end{equation}
\textbf{where $\Omega$ is the inverse mapping 
$\Xi(z) = \xi \Leftrightarrow z = \Omega(\xi)$.} 
\label{corOGF}
\end{corol}

\begin{proof}
\textbf{a)} From the definition (\ref{DefABxNU}) of matrix $B(x,\nu)$, the generating function of the sequence $S = B(x,\nu) \cdot T$ is given by
\begin{align}
\mathfrak{G}_S(z) & \, = \sum_{n \geqslant 1} z^n 
\left ( \sum_{k=1}^n B_{n,k}(x,\nu) T_k \right ) = 
\left ( \sum_{k=1}^n (-1)^k \frac{n!}{k!(n-k)!} F(k-n,k\nu;k;x) T_k \right ) 
\nonumber \\
& \, = \sum_{k \geqslant 1} (-1)^k T_k \frac{z^k}{k!} \sum_{n \geqslant k} 
\frac{n!}{(n-k)!} F(k-n,k\nu;k;x) \, z^{n-k}
\nonumber
\end{align}
after changing the summation order; using the expression (\ref{DefFFbis}) for the Hypergeometric coefficient $F(k-n,k\nu;k;x)$, we then obtain
\begin{align}
\mathfrak{G}_S(z) & = \sum_{k \geqslant 1} (-1)^k T_k \frac{z^k}{k!} 
\sum_{n \geqslant k} \frac{n! \, z^{n-k}}{(n-k)!} 
\nonumber \\
& \quad \sum_{m=0}^{n-k} \frac{(-1)^m(n-k)!}{(n-k-m)!}
\frac{\Gamma(m+k\nu)}{\Gamma(k\nu)} \frac{(k-1)!}{(m+k-1)!}\frac{x^m}{m!}
\nonumber \\
& = \sum_{k \geqslant 1} (-1)^k T_k \frac{z^k}{k} \sum_{n \geqslant k} 
n! \, z^{n-k} \sum_{m=0}^{n-k} \frac{(-1)^m}{(n-k-m)!}
\frac{\Gamma(m+k\nu) \, x^m}{\Gamma(k\nu)m!} \frac{1}{(m+k-1)!}
\nonumber
\end{align}
and the index change $n = k + r$, $r \geqslant 0$, yields
\begin{align}
\mathfrak{G}_S(z) & = \sum_{k \geqslant 1} (-1)^k T_k \frac{z^k}{k} 
\sum_{r \geqslant 0} (k+r)! \, z^{r} \sum_{m=0}^{r} \frac{(-1)^m}{(n-k-m)!}
\frac{(k\nu)_m \, x^m}{m!} \frac{1}{(m+k-1)!}
\nonumber \\
& = \sum_{k \geqslant 1} (-1)^k T_k \frac{z^k}{k} 
\sum_{m \geqslant 0} (-1)^m \, \frac{(k\nu)_m \, x^m}{m!} \frac{1}{(m+k-1)!}
\left ( \sum_{r = m}^{+\infty} \frac{(k+r)!}{(r-m)!} \, z^{r} \right )
\nonumber
\end{align}
where the last sum on index $r$ readily equals
$$
\sum_{r = m}^{+\infty} \frac{(k+r)!}{(r-m)!} \, z^{r} = 
\sum_{r = 0}^{+\infty} \frac{(k+m+r)!}{r!} \, z^{r+m} = 
\frac{(m+k)!}{(1-z)^{k+m+1}} \cdot z^m, \qquad \vert z \vert < 1;
$$
the latest expression of $\mathfrak{G}_S(z)$ consequently reads
\begin{align}
\mathfrak{G}_S(z) = & \, \sum_{k \geqslant 1} 
(-1)^k T_k \frac{z^k}{k} \frac{1}{(1-z)^{k+1}} 
\sum_{m \geqslant 0} (-1)^m \, \frac{(k\nu)_m}{m!} 
\left ( \frac{x\, z}{1-z} \right)^m (m+k)
\nonumber \\
= & \; \frac{1}{1-z} 
\sum_{k \geqslant 1} \frac{T_k}{k} \left ( \frac{z}{z-1} \right )^k
\Bigl [ - \frac{xz}{1-z} 
\sum_{m \geqslant 0} m \left ( \frac{- xz}{1-z} \right )^{m-1} 
\frac{(k\nu)_m}{m!}
\nonumber \\
& \qquad \qquad \qquad \qquad \qquad \; \; + 
k \times \sum_{m \geqslant 0} \left ( \frac{- xz}{1-z} \right )^{m-1} 
\frac{(k\nu)_m}{m!} \Bigr ].
\label{Sum1}
\end{align}
Using successively identity 
$\sum_{m \geqslant 0} (k\nu)_m Z^m/m! = 1/(1-Z)^{k\nu}$ and its term-to-term derivative $\sum_{m \geqslant 0} m (k\nu)_m Z^{m-1}/m! = k\nu/(1-Z)^{k\nu + 1}$ with respect to $Z$, the sum (\ref{Sum1}) reduces to
\begin{align}
\mathfrak{G}_S(z) = & \, \frac{1}{1-z} \left [ \frac{- xz}{1-z} 
\left ( \frac{1-z}{1-(1-x)z} \right ) \nu \cdot \mathfrak{G}_T(\Xi(z)) + \mathfrak{G}_T(\Xi(z)) \right ]
\nonumber \\
= & \, \frac{1}{1-z} \left [ \frac{- \nu xz}{1-(1-x)z} + 1 \right ] 
\mathfrak{G}_T(\Xi(z))
\nonumber
\end{align}
with $\Xi(z)$ defined as in (\ref{DefXi}). Writing 
$$
\frac{1}{1-z} \left [ \frac{- \nu xz}{1-(1-x)z} + 1 \right ] = 
\frac{1-\nu}{1-z} + \frac{\nu}{1-z(1-x)}
$$
eventually entails relation (\ref{eq:OGFRelation}). 

\textbf{b)} For any parameters $x$ and $\nu$, the function 
$z \mapsto \Xi(z)$ is analytic in a neigborhood of $z = 0$, with 
$\Xi(0) = 0$ and $\Xi'(z) \sim -z$ as $z \downarrow 0$, hence 
$\Xi'(0) = -1 \neq 0$. By the Implicit Function Theorem, $\Xi$ has an analytic inverse $\Omega:\xi \mapsto \Omega(\xi)$ in a neighborhood of 
$\xi = 0$ and the inversion of (\ref{eq:OGFRelation}) provides 
(\ref{eq:OGFRelationBIS}), as claimed. 
\end{proof}

\noindent 
Relation (\ref{eq:OGFRelationBIS}) between formal generating series can also be understood as a functional identity between the analytic functions  
$z \mapsto \mathfrak{G}_S(z)$ and $z \mapsto \mathfrak{G}_T(z)$ in some neighborhood of the origin $z = 0$ in the complex plane. Now, Corollary 
\ref{corOGF} can be supplemented by making explicit the inverse mapping 
$\Omega$ involved in the reciprocal relation (\ref{eq:OGFRelationBIS}); to this end, we state some preliminary properties (in the sequel, $\log$ will denote the determination of the logarithm in the complex plane cut along the negative semi-axis $]-\infty,0]$ with $\log(1) = 0$).

\begin{lemma}
\textbf{a) Let $R(\nu) = \vert e^{-\psi(\nu)} \vert$ where}
$$
\psi(\nu) = \left\{
\begin{array}{ll}
(1-\nu)\log(1-\nu) + \nu\log(-\nu), \quad \; 
\nu \in \mathbb{C} \setminus \, [0,+\infty[,
\\ \\
(1-\nu)\log(1-\nu) + \nu\log(\nu), \; \quad \; \; \, 
\nu \in \mathbb{R}, \; 0 \leqslant \nu < 1, 
\\ \\
(1-\nu)\log(\nu-1) + \nu\log(\nu), \; \quad \; \; 
\nu \in \mathbb{R}, \; \nu \geqslant 1.
\end{array} \right.
$$
\textbf{The power series}
$$
\pmb{\Sigma}(w) = 
\sum_{b \geqslant 1} 
\frac{\Gamma(b(1-\nu))}{\Gamma(b)\Gamma(1-b\nu)} \cdot w^b, 
\qquad \vert w \vert < R(\nu),
$$
\textbf{is given by}
\begin{equation}
\pmb{\Sigma}(w) = \frac{\Theta(w)-1}{\nu \, \Theta(w) + 1 - \nu}
\label{U0}
\end{equation}
\textbf{where $\Theta:w \mapsto \Theta(w)$ denotes the unique analytic solution 
(depending on $\nu$) to the implicit equation}
\begin{equation}
1 - \Theta + w \cdot \Theta^{1-\nu} = 0, 
\qquad \vert w \vert < R(\nu),
\label{DefTheta}
\end{equation}
\textbf{verifying $\Theta(0) = 1$.}

\textbf{b) Function $\pmb{\Sigma}$ is the solution to the differential equation}
\begin{align}
w \, \pmb{\Sigma}'(w) = & \, \pmb{\Sigma}(w) 
\left [1 - \nu \, \pmb{\Sigma}(w) \right ]  
\left [1 + (1-\nu)\pmb{\Sigma}(w) \right ]
\nonumber \\
= & \, \pmb{\Sigma}(w) \left [ 1 + (1-2\nu)\pmb{\Sigma}(w) - 
\nu(1-\nu)\pmb{\Sigma}(w)^2 \right ]
\label{EDiff0}
\end{align}
\textbf{with $\pmb{\Sigma}(0) = 0$.}
\label{lemmU}
\end{lemma}

\noindent 
The proof of Lemma \ref{lemmU} is detailed in Appendix \ref{A3}. Quite remarkably, function $\pmb{\Sigma}$ will also prove useful in Section 
\ref{SecA} for the derivation of the generating function of the solution $E$ to the particular system (\ref{T0}). 

\begin{corol}
\textbf{For all $\nu \in \mathbb{C}$ and $x \neq 0$, the inverse mapping 
$\Omega$ of $\Xi$ defined in (\ref{DefXi}) can be expressed by}
\begin{equation}
\Omega(\xi) = 
\frac{\pmb{\Sigma}(x \, \xi)}{(1-x(1-\nu)) \, \pmb{\Sigma}(x \, \xi) - x}, 
\qquad \vert \xi \vert < \frac{R(\nu)}{\vert x \vert},
\label{InvXI}
\end{equation}
\textbf{in terms of power series $\pmb{\Sigma}(\cdot)$ defined in Lemma 
\ref{lemmU}.}
\label{corOGFbis}
\end{corol}

\begin{proof}
\textbf{\textit{(i)}} The homographic transform $M:z \mapsto \theta$ with
$\theta = (1-z)/(1-z(1-x))$ has an inverse $M^{-1}$ defined by
is involutive, that is, with inverse
\begin{equation}
z = M^{-1}(\theta) = \frac{1-\theta}{1-\theta(1-x)}
\label{MobInv}
\end{equation}
(it is thus an involution). Let then $\xi = \Xi(z)$ with function $\Xi$ defined as in (\ref{DefXi}); we first claim that the corresponding 
$\theta = M(z)$ equals $\theta = \Theta(x \, \xi)$ where $\Theta$ is the function defined by the implicit equation (\ref{DefTheta}). In fact, definition (\ref{DefXi}) for $\Xi$ and expression (\ref{MobInv}) for $z$ in terms of $\theta$ together entail
$$
\xi = \Xi(z) = \frac{z}{z-1} \, \theta^{\, \nu} = 
\displaystyle \frac{1-\theta}{1-\theta(1-x)} 
\left ( \frac{1-\theta}{1-\theta(1-x)} - 1 \right )^{-1} \, \theta^{\, \nu} = 
\frac{\theta - 1}{x \, \theta} \, \theta^{\, \nu} 
$$
and the two sides of the latter equalities give 
$1 - \theta + x \xi \theta^{1-\nu} = 0$, hence the identity 
$\theta = \Theta(x \, \xi)$, as claimed. 

\textbf{\textit{(ii)}} The corresponding inverse $z = \Omega(\xi)$ can now be expressed as follows; equality (\ref{U0}) applied to $w = x \, \xi$ can be first solved for $\Theta(x \, \xi)$, giving
$$
\Theta(x \, \xi) = \frac{1 + (1-\nu)\pmb{\Sigma}(x \xi)}
{1 - \nu \, \pmb{\Sigma}(x \xi)};
$$
it then follows from (\ref{MobInv}) and this expression of $\Theta(x \, \xi)$ that
$$
z = \Omega(\xi) = \frac{1 - \Theta(x \, \xi)}{1 - (1-x)\Theta(x \, \xi)} = 
\frac{\displaystyle 1 - \frac{1 + (1-\nu)\pmb{\Sigma}(x \xi)}
{\displaystyle 1 - \nu \, \pmb{\Sigma}(x \xi)}}
{\displaystyle 1 - (1-x)\frac{1 + (1-\nu)\pmb{\Sigma}(x \xi)}
{1 - \nu \, \pmb{\Sigma}(x \xi)}}
$$
which easily reduces to formula (\ref{InvXI}).
\end{proof}


We now turn to the derivation of identities between the exponential generating functions of any pair of related sequences $S$ and $T$.

\begin{corol}
\textbf{Given sequences $S$ and $T$ related by the inversion formulae 
$S = B(x,\nu) \cdot T \Leftrightarrow T = A(x,\nu) \cdot S$, the exponential generating function $\mathfrak{G}_S^*$ of the sequence $S$ can be expressed by}
\begin{equation}
\mathfrak{G}_S^*(z) = \exp(z) \cdot 
\sum_{k \geqslant 1} (-1)^k T_k \, \frac{z^k}{k!} \, \Phi(k\nu;k;-x \, z), 
\qquad z \in \mathbb{C},
\label{GsExp}
\end{equation}
\textbf{$\Phi(\alpha;\beta;\cdot)$ denotes the Confluent Hypergeometric function with parameters $\alpha$, $\beta \notin -\mathbb{N}$.}
\label{corEGF}
\end{corol}

\begin{proof}
A calculation similar to that of Corollary \ref{corOGF} gives
\begin{align}
\mathfrak{G}_S^*(z) = & \, \sum_{n \geqslant 0} \frac{z^n}{n!} 
\left ( \sum_{k=1}^n B_{n,k}(x,\nu) T_k \right )
\nonumber \\
= & \, \sum_{k \geqslant 1} (-1)^k T_k \frac{z^k}{k} 
\sum_{m \geqslant 0} (-1)^m \, \frac{\Gamma(m+k\nu) \, x^m}{\Gamma(k\nu)m!} \frac{1}{(m+k-1)!}
\left ( \sum_{r = m}^{+\infty} \frac{z^r}{(r-m)!} \right );
\nonumber
\end{align}
as $\sum_{r \geqslant m} z^r/(r-m)! = z^m \exp(z)$, the latter reduces to
$$
\mathfrak{G}_S^*(z) = \exp(z) \sum_{k \geqslant 1} (-1)^k T_k \frac{z^k}{k} 
\sum_{m \geqslant 0} (-xz)^m \, \frac{(k\nu)_m}{m!} 
\frac{1}{(k-1)!(k)_m}
$$
which, from the expansion of $\Phi(k\nu;k;-xz)$ in powers of $-xz$, yields  
(\ref{GsExp}).
\end{proof}

\noindent
Expression (\ref{GsExp}), however, does not generally relate to the exponential generating function $\mathfrak{G}_T^*$ of the sequence $T$. We have neither been able to obtain any remarkable identity for the exponential generating function $\mathfrak{G}_T^*$ in terms of $\mathfrak{G}_S^*$. 



\section{Applications}
\label{SecA}


We first apply (Section \ref{SecA1}) the inversion formula of Theorem 
\ref{PropIn} and the associated relations between generating functions 
(Corollaries \ref{corOGF} and \ref{corEGF}) to the resolution of the infinite linear system (\ref{T0}) motivated in the Introduction. Specific extensions of the inversion formula to other families of special polynomials are finally stated (Section \ref{SecA2}).

\subsection{Resolution of infinite system (\ref{T0})}
\label{SecA1}
The resolution of integral equation (\ref{EI0}) has been reduced to that of infinite triangular system (\ref{T0}), whose solution can now be expressed as follows.

\begin{corol}
\textbf{The unique solution $(E_b)_{b \geqslant 1}$ to system (\ref{T0}) is given by}
\begin{equation}
E_b = \frac{1-x}{(U^-)^{b+1}} \sum_{\ell = 1}^b (-1)^{\ell - 1} 
\binom{b}{\ell} F(\ell-b,\ell \nu;\ell;x) \, x^{\ell - 1} 
\frac{\Gamma(\ell - \ell \nu)}{\Gamma(\ell)\Gamma(1-\ell \nu)} \, K_\ell
\label{E0}
\end{equation}
\textbf{for all $b \geqslant 1$.}
\label{C2}
\end{corol}

\begin{proof}
By expression (\ref{Q0}) for the coefficients of lower-triangular matrix $Q$, equation (\ref{T0}) equivalently reads
\begin{equation}
\sum_{\ell=1}^b (-1)^\ell \binom{b}{\ell} F(\ell-b,-b\nu;-b;x) \cdot 
\widetilde{E}_\ell = \widetilde{K}_b, \qquad 1 \leqslant \ell \leqslant b,
\label{S0}
\end{equation}
when setting
\begin{equation}
\left\{
\begin{array}{ll}
\widetilde{E}_\ell = (U^-)^{\ell + 1} E_\ell, 
\qquad \qquad \qquad \qquad \qquad \quad \; \ell \geqslant 1,
\\ \\
\widetilde{K}_b = \displaystyle - \frac{\Gamma(b - b \nu)}
{\Gamma(b)\Gamma(1-b \nu)} 
(1-x)x^{b-1} \cdot K_b, \qquad b \geqslant 1.
\end{array} \right.
\label{S1}
\end{equation}
The application of inversion Theorem \ref{PropIn} to lower-triangular system 
(\ref{S0}) readily provides the solution sequence 
$(\widetilde{E}_\ell)_{\ell \in \mathbb{N}}$ in terms 
of the sequence $(\widetilde{K}_b)_{b \in \mathbb{N}^*}$; using then  transformation (\ref{S1}), the final solution (\ref{E0}) for the sequence 
$(E_\ell)_{\ell \in \mathbb{N}^*}$ follows. 
\end{proof}

The coefficients $K_b$, $b \geqslant 1$, of the right-hand side of system 
(\ref{T0}) can be actually represented by the integral \cite{GQSN18}
\begin{align}
K_b = \int_0^{U^-} \left [ (b-1)(1-\zeta)^b + 1 \right ] \, 
\mathfrak{R}(\zeta)^b \, \frac{\mathrm{d}\zeta}{(1-\zeta)^2}, 
\qquad \quad b \geqslant 1,
\label{CoeffK0}
\end{align}
where $\mathfrak{R}(\cdot)$ is the given function defined by
$$
\mathfrak{R}(\zeta) = \left ( 1 - \frac{\zeta}{U^-} \right )^{-\nu}
\left ( 1 - \frac{\zeta}{U^+} \right )^{\nu - 1}, 
\qquad \zeta \notin [U^-,U^+],
$$
for some real parameters $0 < U^- < U^+$ and $\nu < 0$. From the integral representation (\ref{CoeffK0}) of coefficients $K_b$, $b \geqslant 1$, and as an application of Corollary \ref{corOGF}, the generating function 
$\mathfrak{G}_E$ of the solution $(E_b)_{b \geqslant 1}$ to system (\ref{T0}) can now be given the following integral representation. 

\begin{corol}
\textbf{The generating function $\mathfrak{G}_E$ of the solution 
$(E_b)_{b \geqslant 1}$ to system (\ref{T0}) is given by}
\begin{equation}
\mathfrak{G}_E(z) = \frac{x-1}{x} \left [ \frac{1-\nu}{U^- - z} + 
\frac{\nu}{U^- - (1-x)z} \right ] 
\int_0^{U^-} G \left (\zeta; \frac{z}{U^-} \right ) \, 
\frac{\mathrm{d}\zeta}{(1-\zeta)^2} 
\label{GenE}
\end{equation}
\textbf{with kernel $G$ defined by}
\begin{align}
G(\zeta,z) = & \, 
\pmb{\Sigma} \left[ g \left ( \zeta; z \right ) \right ] \; + 
\nonumber \\
& \, \pmb{\Sigma}^2 
\left[ (1-\zeta) \, g \left( \zeta; z \right ) \right ] 
\left ( 1 - 2\nu - \nu(1-\nu)\pmb{\Sigma}\left[ (1-\zeta) \, 
g \left( \zeta; z \right ) \right ] \right )
\nonumber
\end{align}
\textbf{for small enough $z$, setting 
$g(\zeta;z) = x \, \mathfrak{R}(\zeta) \, \Xi(z)$ and with function 
$\pmb{\Sigma}$ given in Lemma \ref{lemmU}.a).}
\label{C3}
\end{corol}

\begin{proof}
We first calculate the generating function $\mathfrak{G}_{\widetilde{K}}$ of the reduced sequence $(\widetilde{K}_\ell)_{\ell \geqslant 1}$ introduced in 
(\ref{S1}). Using the representation (\ref{CoeffK0}) of $K_\ell$, 
$\ell \geqslant 1$, we have
\begin{align}
& \mathfrak{G}_{\widetilde{K}}(z) = \sum_{\ell \geqslant 1} 
\widetilde{K}_\ell z^\ell \; = 
\nonumber \\
& - \sum_{\ell \geqslant 1} \frac{\Gamma(\ell - \ell \nu)}
{\Gamma(\ell)\Gamma(1-\ell \nu)} (1-x) x^{\ell-1} z^\ell \; 
\int_0^{U^-} \left [ \ell(1-\zeta)^\ell + 1 - (1-\zeta)^\ell \right ] 
\mathfrak{R}(\zeta)^\ell \, \frac{\mathrm{d}\zeta}{(1-\zeta)^2}
\nonumber
\end{align}
that is,
\begin{align}
\mathfrak{G}_{\widetilde{K}}(z) = \frac{x-1}{x} 
\int_0^{U^-} & \frac{\mathrm{d}\zeta}{(1-\zeta)^2} \; 
\Bigl [\sum_{\ell \geqslant 1} \frac{\Gamma(\ell - \ell \nu)} 
{\Gamma(\ell)\Gamma(1-\ell \nu)} 
\nonumber \\
& \left \{ \ell[(1-\zeta)x z]^\ell + (x z)^\ell - 
[(1-\zeta)xz] ^\ell \right \} \, \mathfrak{R}(\zeta)^\ell \Bigr ].
\nonumber
\end{align}
Assume that, for given $x$ and all $\zeta \in [0,U^-]$. $z$ is small enough so that the arguments $\mathfrak{R}(\zeta)(1-\zeta) x z$ and 
$\mathfrak{R}(\zeta) x z$ in the latter integrand together pertain to the open disk centered at the origin and with radius $R(\nu)$, as given in Lemma 
\ref{lemmU}. The series $\pmb{\Sigma}$ introduced in Lemma \ref{lemmU} then enables us to obtain
\begin{align}
\mathfrak{G}_{\widetilde{K}}(z) = \frac{x-1}{x} 
\int_0^{U^-} \frac{\mathrm{d}\zeta}{(1-\zeta)^2} \; 
\Bigl \{ & \; 
\mathfrak{R}(\zeta)(1-\zeta)x z \; 
\pmb{\Sigma}'(\mathfrak{R}(\zeta)(1-\zeta)x z) 
\; + 
\nonumber \\
& \; \pmb{\Sigma}(\mathfrak{R}(\zeta)x z) - 
\pmb{\Sigma}(\mathfrak{R}(\zeta)(1-\zeta)x z) \Bigr \}
\label{GenEbis}
\end{align}
for small enough $z$, where $\pmb{\Sigma}'$ denotes the first derivative of function $\pmb{\Sigma}$; using then the differential equation (\ref{EDiff0}) for the difference $w \, \pmb{\Sigma}'(w) - \pmb{\Sigma}(w)$ applied to the argument $w = \mathfrak{R}(\zeta) (1-\zeta)x z$, formula (\ref{GenEbis})  equivalently reads
\begin{align}
\mathfrak{G}_{\widetilde{K}}(z) = \frac{x-1}{x} 
\int_0^{U^-} & \frac{\mathrm{d}\zeta}{(1-\zeta)^2} \; 
\Bigl \{ \pmb{\Sigma}(\mathfrak{R}(\zeta)x z) \; + 
\nonumber \\ 
& \pmb{\Sigma}^2(\mathfrak{R}(\zeta)(1-\zeta)x z) \cdot 
\left ( 1 - 2\nu - \nu(1-\nu)\pmb{\Sigma}(\mathfrak{R}(\zeta)(1-\zeta)x z) \right ) \Bigr \}
\label{GenEter}
\end{align}
in terms of function $\pmb{\Sigma}$ only. Now, by relation 
(\ref{eq:OGFRelation}) of Corollary \ref{corOGF}, the generating function 
$\mathfrak{G}_{\widetilde{E}}$ of the reduced sequence 
$(\widetilde{E}_\ell)_{\ell \geqslant 1}$ and $\mathfrak{G}_{\widetilde{K}}$ are related by
$$
\mathfrak{G}_{\widetilde{E}}(z) = 
\left [ \frac{1-\nu}{1-z} + \frac{\nu}{1-z(1-x)} \right ] 
\mathfrak{G}_{\widetilde{K}}(\Xi(z)), 
\qquad \vert z \vert < \min \left ( 1, \frac{1}{\vert 1 - x\vert} \right );
$$
using (\ref{GenEter}) in the latter expression and noting from relation 
(\ref{S1}) between sequences $E$ and $\widetilde{E}$ that 
$$
\mathfrak{G}_E(z) = \frac{1}{U^-} \cdot \mathfrak{G}_{\widetilde{E}} 
\left ( \frac{z}{U^-} \right )
$$
for small enough $z$ eventually yields (\ref{GenE}), as claimed. 
\end{proof}

As an application of Corollary \ref{corEGF}, we now derive the exponential generating function of the solution $(E_b)_{b \geqslant 1}$. Note that the notation $\mathfrak{G}^*_E(z)$ for the generating function of this sequence 
$(E_b)_{b \geqslant 1}$ used below is equivalent to the notation $E^*(z)$ introduced in (\ref{defE*}) for the entire solution to integral equation 
(\ref{EI0}).

\begin{corol}
\textbf{For $x \neq 0$ and $0 < \mathrm{Re}(\nu) < 1$, the exponential generating function $\mathfrak{G}_E^*$ of the solution 
$(E_b)_{b \geqslant 1}$ to system (\ref{T0}) can be given the double integral representation}
\begin{equation}
\mathfrak{G}_E^*(z) = \frac{1-x}{\pi \, x \, U^-} \, 
\exp \left ( \frac{z}{U^-} \right ) 
\int_0^{U^-} \frac{\mathrm{d}\zeta}{(1-\zeta)^2} 
\int_0^1 H_\nu \left( \zeta,t; \frac{z}{U^-} \right ) \frac{\mathrm{d}t}{1-t}
\label{GEK0}
\end{equation}
\textbf{with kernel $H$ defined by}
\begin{align}
H_\nu(\zeta,t;z) = & \, 
e^{-\cos(\nu\pi)h_\nu(\zeta,t) \, z}
\sin \Bigl [ \sin(\nu\pi) \, h_\nu(\zeta,t) \, z \Bigr ] \; - 
\nonumber \\
& \, e^{-\cos(\nu\pi)(1-\zeta)h_\nu(\zeta,t) \, z}
\cdot 
\Bigl \{ \sin \Bigl [ \sin(\nu\pi) \, (1-\zeta) \, h_\nu(\zeta,t) \, z \Bigr ] 
\; +
\nonumber \\
& \, (1-\zeta) \, h_\nu(\zeta,t) \, z \cdot 
\sin \Bigl [ \nu \pi - \sin(\nu\pi) \, (1-\zeta) \, h_\nu(\zeta,t) \, z \Bigr ] \Bigr \}
\nonumber
\end{align}
\textbf{for all $z \in \mathbb{C}$, where we set 
$h_\nu(\zeta,t) = x \, \mathfrak{R}(\zeta) \cdot t^\nu(1-t)^{1-\nu}$.}
\label{C4}
\end{corol}

\begin{proof}
Using the integral representation of the Confluent Hypergeometric function 
(\cite{NIST10}, Chap.13, 13.4.1), write
$$
\Phi(b\nu;b;-xz) = \frac{\Gamma(b)}{\Gamma(b\nu)\Gamma(b(1-\nu))} 
\int_0^1 e^{-xz \, t} t^{b\nu-1}(1-t)^{b(1-\nu)-1} \, \mathrm{d}t, 
\qquad z \in \mathbb{C},
$$
for all $b \in \mathbb{N}^*$ and with $0 < \mathrm{Re}(\nu) < 1$; applying then relation (\ref{GsExp}) between sequences $S = \widetilde{E}$ and 
$T = \widetilde{K}$, on account of formula (\ref{S1}) for $\widetilde{K}$ in terms of sequence $K$, we obtain
\begin{align}
\mathfrak{G}^*_{\widetilde{E}}(z) = \exp(z) \sum_{b \geqslant 1} 
(-1)^b \frac{z^b}{b!} (-1) & \frac{\Gamma(b(1-\nu))}{\Gamma(b)\Gamma(1-b\nu)}
(1-x)x^{b-1} \cdot K_b \; \times 
\nonumber \\
& \, \frac{\Gamma(b)}{\Gamma(b\nu)\Gamma(b(1-\nu))} 
\int_0^1 e^{-xz \, t} t^{b\nu-1}(1-t)^{b(1-\nu)-1} \, \mathrm{d}t
\label{GEK0bis}
\end{align}
which, after the reflection formula $\Gamma(a)\Gamma(1-a) = 
\pi/\sin(\pi a)$ (\cite{NIST10}, §5.5.3) applied to the argument $a = b\nu$, reads
$$
\mathfrak{G}^*_{\widetilde{E}}(z) = (1-x) e^z \sum_{b \geqslant 1} 
(-1)^{b-1} \frac{z^b}{b!} \, \frac{\sin(b\nu\pi)}{\pi} \, x^{b-1} K_b 
\int_0^1 e^{-xz \, t} t^{b\nu-1}(1-t)^{b(1-\nu)-1} \, \mathrm{d}t.
$$
Now, using the integral representation of the sequence 
$K = (K_b)_{b \geqslant 1}$ given in (\ref{CoeffK0}) and inverting the integration (in both variables $\zeta$ and $t$) and series summation orders, the latter identity for $\mathfrak{G}^*_{\widetilde{E}}(z)$ yields
\begin{align}
\mathfrak{G}^*_{\widetilde{E}}(z) = & \, (1-x) \exp(z) \int_0^{U^-} 
\frac{\mathrm{d}\zeta}{(1-\zeta)^2} \int_0^1 \frac{\mathrm{d}t}{1-t} \; \times
\nonumber \\
& \, \sum_{b \geqslant 1} (-1)^{b-1}\frac{z^b}{b!} 
\frac{\sin(b\nu\pi)}{\pi} \, x^{b-1} \left [ (b-1)(1-\zeta)^b + 1 \right ]
\left [ \mathfrak{R}(\zeta)\, t^\nu (1-t)^{1-\nu} \right ] ^b.
\label{GEK1}
\end{align}
Writing $\sin(b\nu\pi) = (e^{ib\nu\pi} - e^{-ib\nu\pi})/2i$, $i^2 = -1$, we easily obtain the formulas
$$
\left\{
\begin{array}{ll}
\displaystyle \sum_{b \geqslant 1} (-1)^{b-1} \sin(b\nu\pi) \frac{W^b}{b!} = 
e^{-\cos(\nu\pi)W}\sin(\sin(\nu\pi)W), 
\\ \\
\displaystyle 
\sum_{b \geqslant 1} (-1)^{b-1} b \sin(b\nu\pi) \frac{W^b}{b!} = W 
e^{-\cos(\nu\pi)U}\sin(\nu\pi - \sin(\nu\pi)W)
\end{array} \right.
$$
(the latter following by differentiation of the former with respect to variable $W$); applying these formulas to the summation of the series in expression (\ref{GEK1}) (when successively setting 
$W = x \, \mathfrak{R}(\zeta) \, t^\nu (1-t)^{1-\nu}z$ and 
$W = x (1-\zeta) \mathfrak{R}(\zeta) \, t^\nu (1-t)^{1-\nu}z$) then gives 
\begin{equation}
\mathfrak{G}^*_{\widetilde{E}}(z) = \frac{1-x}{\pi \, x} \, \exp(z) 
\int_0^{U^-} \frac{\mathrm{d}\zeta}{(1-\zeta)^2} 
\int_0^1 \frac{\mathrm{d}t}{1-t} \, H_\nu(\zeta,t;z) , \quad z \in \mathbb{C},
\label{GEK2}
\end{equation}
with $H_\nu(\zeta,t;z)$ given as in the Corollary. Noting from relation 
(\ref{S1}) between sequences $E$ and $\widetilde{E}$ that 
$$
\mathfrak{G}_E^*(z) = 
\frac{1}{U^-} \cdot 
\mathfrak{G}_{\widetilde{E}}^* \left ( \frac{z}{U^-} \right ), 
\quad z \in \mathbb{C},
$$
eventually yields the final representation (\ref{GEK0}), as claimed.
\end{proof}

The integral representation of $\mathfrak{G}_E^*$ obtained in Corollary 
\ref{C4} in the case when $0 < \mathrm{Re}(\nu) < 1$ can be extended to a larger domain of values of parameter $\nu$, provided that the integral w.r.t. variable $t$ is replaced by a contour integral in the complex plane. In this manner, we can assert the following.

\begin{corol}
\textbf{For $x \neq 0$ and $\mathrm{Re}(\nu) < 1$, the exponential generating function $\mathfrak{G}_E^*$ of the solution $(E_b)_{b \geqslant 1}$ to 
(\ref{T0}) can be given the double integral representation}
\begin{equation}
\mathfrak{G}^*_{E}(z) = \frac{1-x}{2i\pi \, x \, U^-} \, e^{(1-x)z} 
\int_0^{U^-} \frac{\mathrm{d}\zeta}{(1-\zeta)^2} 
\int_0^{(1)^+} J_\nu \left ( \zeta,t; \frac{z}{U^-}\right ) 
\frac{\mathrm{d}t}{t(t-1)} 
\label{GEK2bis}
\end{equation}
\textbf{with kernel $J_\nu$ defined by}
$$
J_\nu(\zeta,t;z) = 
\left [ 1 + (1-\zeta) j_\nu(\zeta,t) z \right ] 
e^{-(1-\zeta) j_\nu(\zeta,t) \, z} -  e^{-j_\nu(\zeta,t)z}
$$
\textbf{for all $z \in \mathbb{C}$, where we set 
$j_\nu(\zeta,t) = x \, \mathfrak{R}(\zeta) \cdot t^{1-\nu}(t-1)^\nu$.}

\textbf{(The contour in integral (\ref{GEK2bis}) in variable $t$ is a loop starting and ending at point $t = 0$, and encircling point $t = 1$ once in the positive sense).}
\label{C5}
\end{corol}

\begin{proof}
Invoke the Kummer transformation (\cite{NIST10}, Chap.13, 13.2.39) to write 
\begin{equation}
\Phi(b\nu;b;-xz) = e^{-xz} \Phi(b(1-\nu);b;xz), \qquad b \in \mathbb{N}^*,
\label{IRK0}
\end{equation}
together with the integral representation of the Confluent Hypergeometric function (\cite{NIST10}, Chap.13, 13.4.9)
\begin{equation}
\Phi(\alpha;\beta;Z) = \frac{1}{2 i \pi}
\frac{\Gamma(1+\alpha-\beta)\Gamma(\beta)}{\Gamma(\alpha)} 
\int_0^{(1)^+} e^{Zt}t^{\alpha - 1}(t-1)^{\beta-\alpha-1} \, \mathrm{d}t
\label{IRK1}
\end{equation}
for $\beta-\alpha \notin -\mathbb{N}$ and $\mathrm{Re}(\alpha) > 0$. On account of (\ref{IRK0}) and (\ref{IRK1}) applied to $\alpha = b(1-\nu)$ for 
$\mathrm{Re}(\nu) < 1$ and $\beta = b \in \mathbb{N}^*$, relation 
(\ref{GsExp}) between sequences $S = \widetilde{E}$ and $T = \widetilde{K}$ now reads
\begin{align}
\mathfrak{G}^*_{\widetilde{E}}(z) = \exp(z) \cdot \sum_{b \geqslant 1} 
(-1)^b & \frac{z^b}{b!} (-1) \frac{\Gamma(b(1-\nu))}{\Gamma(b)\Gamma(1-b\nu)}
(1-x)x^{b-1} \cdot K_b \; \times 
\nonumber \\
& \; \; \; \frac{e^{-xz}}{2i\pi} 
\frac{\Gamma(1-b\nu))\Gamma(b)}{\Gamma(b(1-\nu))} 
\int_0^{(1)^+} e^{xz \, t} t^{b(1-\nu)-1}(1-t)^{b\nu-1} \, \mathrm{d}t
\label{GEK1bis}
\end{align}
for $\mathrm{Re}(\nu) < 1$; all factors depending on the $\Gamma$ function in 
(\ref{GEK1bis}) cancel out and the latter reduces to
$$
\mathfrak{G}^*_{\widetilde{E}}(z) = \frac{(1-x)}{2i\pi \, x} e^{(1-x)z} 
\sum_{b \geqslant 1} 
(-1)^{b-1} \frac{z^b}{b!} \, x^{b-1} K_b 
\int_0^{(1)^+} e^{xz \, t} t^{b(1-\nu)-1}(1-t)^{b\nu-1} \, \mathrm{d}t.
$$
Using the integral representation (\ref{CoeffK0}) of the sequence 
$K = (K_b)_{b \geqslant 1}$ and performing the exponential series summations  then easily yields formula (\ref{GEK2bis}) for $\mathfrak{G}_E^*$.
\end{proof}

\subsection{Consequences of the inversion formulas}
\label{SecA2}
We now show how matrices involving other special polynomials can be recast into our general inversion scheme (\ref{eq:inversionR}). Let 
$L_n^{(\alpha)}(x)$ denote the generalized Laguerre polynomial with order 
$n \in \mathbb{N}$ and parameter $\alpha \in \mathbb{C}$.

\begin{corol}
\textbf{Let $x \in \mathbb{C}$ and define the lower-triangular matrices 
$\widetilde{A}(x)$ and $\widetilde{B}(x)$ by}
	\begin{equation}
	\left\{
	\begin{array}{ll}
	\widetilde{A}_{n,k}(x) = \displaystyle (-1)^n L_{n-k}^{(-n-1)}(-nx),
  \\ \\
	\widetilde{B}_{n,k}(x) = \displaystyle (-1)^k \, \frac{n}{k} \, 
	L_{n-k}^{(k-1)}(kx)
	\end{array} \right.
	\label{DefABx}
	\end{equation}
\textbf{for $1 \leqslant k \leqslant n$. For any pair of complex sequences 
$(S_n)_{n \in \mathbb{N}^*}$ and $(T_n)_{n \in \mathbb{N}^*}$, the inversion formula}
	\begin{equation}
	T_n = \sum_{k=1}^{n}\widetilde{A}_{n,k}(x)S_k 
	\Longleftrightarrow 
	S_n=\sum_{k=1}^{n}\widetilde{B}_{n,k}(x)T_k, \quad n \in \mathbb{N}^*,
	\label{eq:inversionRBIS}
	\end{equation}
\textbf{holds.}
	\label{PropInBIS}
\end{corol}

\begin{proof}
Applying the substitution $x \in \mathbb{C} \mapsto x/\nu$ in definition   
(\ref{DefABxNU}), and using the fact that $(-n\nu)_j \sim (-n)^j \nu^j$ for large $\nu$ and given $j \geqslant 1$, expression (\ref{DefFpoly}) gives 
$$
F \left ( k-n,-n\nu;-n;\frac{x}{\nu} \right ) = \sum_{j = 0}^{n-k}
\frac{(k-n)_j(-n\nu)_j}{(-n)_j}\left ( \frac{x}{\nu} \right )^j \frac{1}{j!} 
\longrightarrow \sum_{j = 0}^{n-k} 
\frac{(k-n)_j(-n)^j}{(-n)_j} \, \frac{x^j}{j!}
$$
when $\nu \rightarrow \infty$. This limit statement equivalently reads
$$
\lim_{\nu \rightarrow \infty} F \left ( k-n,-n\nu;-n;\frac{x}{\nu} \right ) = 
M(k-n,-n,-nx)
$$
for given $k$, $n$ and $x$, $M(a,b;\cdot)$ denoting the first Kummer function with parameters $a$ and $b$ (\cite{NIST10}, 13.2.2); in the present case of a negative integer parameter $a = k-n$, the Kummer function $M(k-n,-n,\cdot)$  further relates to the generalized Laguerre polynomial $L_{n-k}^{(-n-1)}$ by the identity (\cite{NIST10}, 13.6.19)
\begin{equation}
L_{n-k}^{(-n-1)}(-nx) = \binom{n-k-(n+1)}{n-k} \times M(k-n,-n,-nx)
\label{Lagu1}
\end{equation}
where the latter binomial coefficient simply reduces to 
$(-1)^{n-k} \binom{n}{k}$ after elementary manipulation. From the latter discussion and identity (\ref{Lagu1}), we therefore derive that the scaled coefficient $A_{n,k}(x/\nu,\nu)$ has the limit
\begin{align}
\lim_{\nu \rightarrow \infty} 
A_{n,k} \left ( \frac{x}{\nu}, \nu \right ) & \, =  
\widetilde{A}_{n,k}(x) = (-1)^k \binom{n}{k} \times 
M(k-n,-n,-nx) 
\nonumber \\
& \, = (-1)^n L_{n-k}^{(-n-1)}(-nx)
\label{Lagu2}
\end{align}
with given $1 \leqslant k \leqslant n$ and complex $x$. In a similar manner, definition (\ref{DefABxNU}) entails that the scaled coefficient 
$B_{n,k}(x/\nu,\nu)$ has the limit
$$
\lim_{\nu \rightarrow \infty} B_{n,k} \left ( \frac{x}{\nu}, \nu \right ) = 
(-1)^k \binom{n}{k}M(k-n,k,kx)
$$
where $M(k-n,k,\cdot)$ relates in turn to the Laguerre polynomial 
$L_{n-k}^{(k-1)}$ via 
\begin{equation}
L_{n-k}^{(k-1)}(kx) = \binom{n-k+(k-1)}{n-k} \times M(k-n,k,kx)
\label{Lagu3}
\end{equation}
and where the binomial coefficient reduces to $k \binom{n}{k}/n$. From the previous results and identity (\ref{Lagu3}), we deduce that the scaled coefficient $B_{n,k}(x/\nu,\nu)$ tends to
\begin{equation}
\lim_{\nu \rightarrow \infty} B_{n,k} \left ( \frac{x}{\nu}, \nu \right ) =  \widetilde{B}_{n,k}(x) = 
(-1)^k \, \frac{n}{k} \, L_{n-k}^{(k-1)}(kx)
\label{Lagu4}
\end{equation}
for given $1 \leqslant k \leqslant n$ and complex $x$. Inversion formulae 
(\ref{eq:inversionRBIS}) with the required matrices $\widetilde{A}(x)$ and 
$\widetilde{B}(x)$ consequently follow.
\end{proof}




\section{Appendix}


\subsection{Proof of Lemma \ref{lemm1}}
\label{A1}
\textbf{a)} By the reflection formula 
$\Gamma(z)\Gamma(1-z) = \pi/\sin(\pi \, z)$, $z \notin -\mathbb{N}$ 
(\cite{NIST10}, §5.5.3), applied to the argument $z = r-\mu$, the generic term $d_r(\lambda,\mu)$ of the sum $D_N(\lambda,\mu)$ equivalently reads
$$
d_r(\lambda,\mu) = \frac{(-1)^r}{\Gamma(1+r-\lambda)\Gamma(1-r+\mu)} = 
- \frac{\sin(\pi \mu)}{\pi} \, \frac{\Gamma(r-\mu)}{\Gamma(1+r-\lambda)}
$$
and Stirling's formula (\cite{NIST10}, §5.11.3) for large $r$ entails that 
$d_r(\lambda,\mu) = O(r^{\lambda-\mu-1})$; the series 
$\sum_{r \geqslant 0} d_r(\lambda,\mu)$ is therefore convergent if and only if 
$\Re(\mu) > \Re(\lambda)$. Write then the finite sum $D_N(\lambda,\mu)$ as the difference
\begin{align}
& \, \sum_{r=0}^{+\infty} \frac{(-1)^r}{\Gamma(1+r-\lambda)\Gamma(1-r+\mu)} - 
\sum_{r=N}^{+\infty}\frac{(-1)^r}{\Gamma(1+r-\lambda)\Gamma(1-r+\mu)} \; = 
\nonumber \\
& \, \sum_{r=0}^{+\infty}\frac{(-1)^r}{\Gamma(1+r-\lambda)\Gamma(1-r+\mu)} - 
\sum_{r=0}^{+\infty}\frac{(-1)^{r+N}}{\Gamma(1+r+N-\lambda)\Gamma(1-r-N+\mu)};
\nonumber
\end{align}
applying similarly the reflection formula to the argument $z = r-\mu+N$ for the second sum, we obtain
\begin{align}
D_N(\lambda,\mu) & \, = \frac{\sin(\pi \, \mu)}{\pi} 
\left [ \sum_{r=0}^{+\infty} \frac{\Gamma(r-\mu+N)}{\Gamma(1+r+N-\lambda)} 
- \sum_{r=0}^{+\infty} \frac{\Gamma(r-\mu)}{\Gamma(1+r-\lambda)} \right ]
\nonumber \\
& \, = \frac{\sin(\pi \, \mu)}{\pi} 
\left [ \sum_{r=0}^{+\infty} 
\frac{(N-\mu)_r\Gamma(r-\mu)}{(1+N-\lambda)_r\Gamma(1+N-\lambda)} 
- \sum_{r=0}^{+\infty} \frac{(-\mu)_r\Gamma(-\mu)}{(1-\lambda)_r\Gamma(1-\lambda)} \right ]
\nonumber
\end{align}
when introducing Pochhammer symbols of order $r$, hence
\begin{align}
D_N(\lambda,\mu) = \frac{\sin(\pi \, \mu)}{\pi} 
\Bigl [ & \, \frac{\Gamma(N-\mu)}{\Gamma(1+N-\lambda)} \, F(1,N-\mu;1+N-\lambda;1) \; - 
\nonumber \\
& \, \frac{\Gamma(-\mu)}{\Gamma(1-\lambda)} \, F(1,-\mu;1-\lambda;1) \Bigr ]
\nonumber
\end{align}
after the definition of the Hypergeometric function $F$. Now, recall the identity (\cite{GRAD07}, §9.122.1)
\begin{equation}
F(\alpha,\beta;\gamma;1) = \frac{\Gamma(\gamma)\Gamma(\gamma-\alpha-\beta)}
{\Gamma(\gamma-\alpha)\Gamma(\gamma-\beta)}, \qquad 
\Re(\gamma) > \Re(\alpha + \beta);
\label{HyperF1}
\end{equation}
when aplying (\ref{HyperF1}) to the values $\alpha = 1$, $\beta = N-\mu$, 
$\gamma = 1 + N -\lambda$ (resp. $\alpha = 1$, $\beta = -\mu$, 
$\gamma = 1 - \lambda$), the latter sum $D_N(\lambda,\mu)$ consequently reduces to 
\begin{equation}
D_N(\lambda,\mu) = \frac{\sin(\pi \, \mu)}{\pi} 
\frac{\Gamma(\mu-\lambda)}{\Gamma(1-\lambda+\mu)} 
\left [ \frac{\Gamma(N-\mu)}{\Gamma(N-\lambda)} - \frac{\Gamma(-\mu)}{\Gamma(-\lambda)} \right ], \qquad \Re(\mu) > \Re(\lambda).
\label{A11}
\end{equation}
By the reflection formula for function $\Gamma$ again, we have
$$
\Gamma(N-\mu)\Gamma(1-N+\mu) = - \frac{(-1)^N \pi}{\sin(\pi \mu)}, \qquad 
\Gamma(-\mu)\Gamma(1+\mu) = - \frac{\pi}{\sin(\pi \mu)},
$$
so that expression (\ref{A11}) eventually yields 
\begin{align}
D_N(\lambda,\mu) & \, = - \frac{\Gamma(\mu-\lambda)}{\Gamma(1-\lambda+\mu)} 
\left [ \frac{(-1)^N}{\Gamma(N-\lambda)\Gamma(1-N+\mu)} - 
\frac{1}{\Gamma(-\lambda)\Gamma(1+\mu)} \right ] 
\nonumber \\
& \, = \frac{1}{\lambda-\mu} 
\left [ \frac{(-1)^N}{\Gamma(N-\lambda)\Gamma(1-N+\mu)} - 
\frac{1}{\Gamma(-\lambda)\Gamma(1+\mu)} \right ]
\nonumber
\end{align}
which states the first identity (\ref{Sn}) for $\Re(\mu) > \Re(\lambda)$.

\textbf{b)} Besides, the reflection formula of function $\Gamma$ applied to 
$z = r-\lambda$ enables us to write $D_N(\lambda,\lambda)$ as
\begin{align}
D_N(\lambda,\lambda) & \, = 
\sum_{r=0}^{N-1} \frac{(-1)^r}{\Gamma(1+r-\lambda)\Gamma(1-r+\lambda)} 
= - \frac{\sin(\pi \lambda)}{\pi} \sum_{r=0}^{N-1} 
\frac{\Gamma(r-\lambda)}{\Gamma(1-r+\lambda)} 
\nonumber \\
& \, = - \frac{\sin(\pi \lambda)}{\pi} \sum_{r=0}^{N-1} \frac{1}{r-\lambda} 
= \frac{\sin(\pi \lambda)}{\pi} 
\left [ \psi(-\lambda) - \psi(N-\lambda) \right ]
\nonumber
\end{align}
after the expansion formula (\cite{NIST10}, Chap.5, §5.7.6) for function 
$\psi$ and the second identity (\ref{Sn}) for $\mu = \lambda$ follows.

\textbf{c)} The first identity (\ref{Sn}) stated for 
$\Re(\mu) > \Re(\lambda)$ defines an analytic function of variables 
$\lambda \in \mathbb{C}$ and $\mu \in \mathbb{C}$ for $\mu \neq \lambda$; besides, it is easily verified that this function has the limit given by 
$D_N(\lambda,\lambda)$ when $\mu \rightarrow \lambda$. On the other hand, the finite sum $D_N(\lambda,\mu)$ defines itself an entire function of 
$\lambda \in \mathbb{C}$ and $\mu \in \mathbb{C}$; by analytic continuation, identity (\ref{Sn}) consequently holds for any pair 
$(\lambda,\mu) \in \mathbb{C} \times \mathbb{C}$ $\blacksquare$

\subsection{Proof of Lemma \ref{lemmU}}
\label{A3}
\textbf{a)} We first determine the convergence radius of the power series 
$\pmb{\Sigma}(w)$ in terms of complex parameter $\nu$. For large $b$, 

$\bullet$ if $1 - \nu \notin \; ]-\infty,0]$ and $-\nu \notin \; ]-\infty,0]$, that is, if $\nu \in \mathbb{C} \setminus [0,+\infty[$, the generic term 
$\sigma_b$ of this series is asymptotic to
$$
\sigma_b = \frac{\Gamma(b(1-\nu))}{\Gamma(b)\Gamma(1-b\nu)} 
= - \frac{1}{\nu} \cdot \frac{\Gamma(b(1-\nu))}{b! \, \Gamma(-b\nu)} 
\sim - \sqrt{\frac{-\nu}{2\pi(1-\nu)b}} \, e^{b \cdot \varphi^-(\nu)}
$$
after Stirling's formula $\Gamma(z) \sim \sqrt{2\pi} e^{z \log z -z}/\sqrt{z}$ for large $z$ with $\vert \mathrm{arg}(z) \vert \leqslant \pi - \delta$, 
$\delta > 0$ (\cite{NIST10}, Chap.5, 5.11.3) , and where we set  
$\varphi^-(\nu) = (1-\nu) \log (1-\nu) + \nu \log(-\nu)$; 

$\bullet$ if $1 - \nu \notin \; ]-\infty,0]$ and $\nu \in [0,+\infty[$ (the parameter $\nu$ is consequently real), that is, $0 \leqslant \nu < 1$, write 
$\Gamma(1-b\nu) = \pi / [\sin(\pi b \nu) \Gamma(b\nu)]$ after the reflection formula so that the generic term $\sigma_b$ is now asymptotic to
$$
\sigma_b = - \frac{1}{\nu} \cdot \frac{\Gamma(b(1-\nu))}{b! \, \pi} 
\Gamma(b\nu) \, \sin(\pi b \nu) \sim 
- \frac{1}{\nu} \sqrt{\frac{\pi}{2\nu(1-\nu)b^3}} \, \sin(\pi b \nu) 
\, e^{b \cdot \varphi(\nu)}
$$
after Stirling's formula (ibid.) and where
$\varphi(\nu) = (1-\nu) \log (1-\nu) + \nu \log(\nu)$;

$\bullet$ finally if $\nu - 1 \in [0,+\infty]$, that is, if $\nu \geqslant 1$, write $\Gamma(1-b\nu) = \pi / [\sin(\pi b \nu) \Gamma(b\nu)]$ together with  
$\Gamma(1-b(1-\nu)) = \pi / [\sin(\pi b (1-\nu)) \Gamma(b(1-\nu))]$ after the reflection formula so that the generic term $\sigma_b$ is asymptotic to
$$
\sigma_b = \frac{(-1)^{b-1}}{\nu} \cdot 
\frac{\Gamma(b\nu)}{b! \, \Gamma(1-b(1-\nu))} \sim 
\frac{(-1)^{b-1}}{\nu} \sqrt{\frac{1}{2\pi \, \nu(\nu - 1)b^3}} \, 
e^{b \cdot \varphi^+(\nu)}
$$
after Stirling's formula and where
$\varphi^+(\nu) = (1-\nu) \log (\nu-1) + \nu \log(\nu)$. By the latter discussion, it therefore follows that the power series $\pmb{\Sigma}(w)$ has the finite convergence radius $R(\nu) = \vert e^{-\psi(\nu)} \vert$ with 
$\psi(\nu)$ defined by 
$$
\psi(\nu) = \left\{
\begin{array}{ll}
\varphi^-(\nu), \quad \; \, 
\nu \in \mathbb{C} \setminus \, [0,+\infty[,
\\ \\
\varphi(\nu), \; \quad \; \; \, \, 
\nu \in \mathbb{R}, \; 0 \leqslant \nu < 1, 
\\ \\
\varphi^+(\nu), \; \quad \; 
\nu \in \mathbb{R}, \; \nu \geqslant 1,
\end{array} \right.
$$
as given in Lemma \ref{lemmU}.

Now, by the above expression of $\sigma_b$ for $\nu \in \mathbb{C} \setminus 
[0,+\infty[$, write
\begin{equation}
\sigma_b = - \frac{1}{\nu} \cdot \frac{\Gamma(b(1-\nu))}{b! \, \Gamma(-b\nu)}  = - \frac{1}{\nu} \cdot \binom{-1 + b(1-\nu)}{b} = - \frac{1}{\nu} \cdot 
\binom{\alpha + b\beta}{b}
\label{DefSig0}
\end{equation}
for all $b \geqslant 1$, where we set $\alpha = -1$ and $\beta = 1-\nu$. From 
(\cite{PolSze72}, Problem 216, p.146, p. 349), it is known that 
\begin{equation}
1 + \sum_{b \geqslant 1} \binom{\alpha + b\beta}{b} w^b = 
\frac{\Theta(w)^{\alpha+1}}{(1-\beta)\Theta(w) + \beta}
\label{PolSzeg0}
\end{equation}
for any pair $\alpha$ and $\beta$, where $\Theta(w)$ denotes the unique solution to the implicit equation $1 - \Theta + w\, \Theta^\beta = 0$
with $\Theta(0) = 1$. By expression (\ref{DefSig0}) and relation 
(\ref{PolSzeg0}) applied to the specific values $\alpha = -1$ and 
$\beta = 1-\nu$, we can consequently assert that the series $\pmb{\Sigma}(w)$ equals
$$
\pmb{\Sigma}(w) = \sum_{b \geqslant 1} \sigma_b \, w^b = - \frac{1}{\nu} 
\left [ \frac{1}{\nu \, \Theta(w) + 1-\nu} - 1 \right ] = 
\frac{\Theta(w)-1}{\nu \, \Theta(w) + 1-\nu}
$$
for $\vert w \vert < R(\nu)$, as claimed. The validity of equality 
(\ref{U0}) for real $\nu \in [0,+\infty[$ follows by analytic continuation. 

\textbf{b)} By differentiating the implicit relation (\ref{DefTheta}) at point $w \neq 0$, we obtain the equality 
$-\Theta'(w) + \Theta(w)^{1-\nu} 
+ w \, \Theta'(w) \Theta(w)^{-\nu}(1-\nu) = 0$, hence 
\begin{align}
\Theta'(w) = & \, \frac{\Theta(w) ^{1-\nu}}{1 - w \, \Theta(w)^{-\nu}(1-\nu)} 
= \frac{(\Theta(w)-1)/w}{1 - w(\Theta(w)-1)(1-\nu)/w\Theta(w)} 
\nonumber \\
= & \, \frac{\Theta(w)}{w} \, \frac{\Theta(w)-1}{\nu \, \Theta(w) + 1 - \nu}
\nonumber 
\end{align}
after using relation (\ref{DefTheta}) again for $\Theta(w)^{1-\nu}$; using  relation (\ref{U0}), the latter expression for $\Theta'(w)$ consequently reduces to 
\begin{equation}
\Theta'(w) = \frac{\Theta(w)}{w} \, \pmb{\Sigma}(w).
\label{Theta1}
\end{equation}
Now, differentiating (\ref{U0}) at point $w$ and using (\ref{Theta1}) yields
\begin{equation}
\pmb{\Sigma}'(w) = \frac{\Theta'(w)}{(\nu \, \Theta(w) + 1 - \nu)^2} = 
\frac{\Theta(w) \, \pmb{\Sigma}(w)}{w(\nu \, \Theta(w) + 1 - \nu)^2};
\label{Theta2}
\end{equation}
but solving (\ref{U0}) for $\Theta(w)$ in terms of $\pmb{\Sigma}(w)$ readily gives the rational expressions
$$
\Theta(w) = \frac{(1-\nu)\pmb{\Sigma}(w) + 1}{1-\nu \, \pmb{\Sigma}(w)}, 
\qquad \nu \Theta(w) + 1 - \nu = \frac{1}{1 - \nu \pmb{\Sigma}(w)}
$$ 
which, once replaced into the right-hand side of (\ref{Theta2}), entail
$$
\pmb{\Sigma}'(w) = 
\frac{\displaystyle \frac{(1-\nu)\pmb{\Sigma}(w)+1}{1-\nu \, \pmb{\Sigma}(w)} \times 
\pmb{\Sigma}(w)}{w \displaystyle \left ( \frac{1}{1 - \nu \pmb{\Sigma}(w)} \right )^2}
$$
and readily provide differential equation (\ref{EDiff0}) after algebraic reduction $\blacksquare$


\end{document}